\documentclass[11pt,letterpaper,reqno]{amsart}
\usepackage{tikz}
\usetikzlibrary{positioning, shapes.geometric, arrows.meta}
\usepackage{amssymb}
\usepackage{amsmath}
\usepackage{amsthm}
\usepackage{amsfonts}
\usepackage{bbm}
\usepackage{enumitem} 
\usepackage{pgfplots}
\pgfplotsset{compat=1.18} 

\usepackage{graphicx}
\usepackage[T1]{fontenc}
\usepackage{doi}
\usepackage{float} 
\usepackage{bookmark}
\usepackage{hyperref}
\hypersetup{pdfstartview={FitH}}

\DeclareMathOperator{\rank}{rank}
\DeclareMathOperator{\msr}{msr}

\newtheorem{thm}{Theorem}[section]
\newtheorem{lem}[thm]{Lemma}
\newtheorem{prop}[thm]{Proposition}

\theoremstyle{definition}

\newtheorem{remark}[thm]{Remark}

\numberwithin{equation}{section}

\makeatother

\begin{document}
\title{Nordhaus--Gaddum type bounds for the complement rank}

\author[Q.~Tang]{Quanyu Tang}
\address{School of Mathematics and Statistics, Xi'an Jiaotong University, Xi'an 710049, P. R. China}
\email{tang\_quanyu@163.com}

% 05C50      Graphs and linear algebra (matrices, eigenvalues, etc.)
% 05C35  	 Extremal problems in graph theory


\begin{abstract}
Let $G$ be an $n$-vertex simple graph with adjacency matrix $A_G$. 
The \emph{complement rank} of $G$ is defined as $\operatorname{rank}(A_G+I)$, where $I$ is the identity matrix. 
In this paper we study Nordhaus--Gaddum type bounds for the complement rank. We prove that for every graph $G$,
$$
\operatorname{rank}(A_G+I)\cdot\operatorname{rank}(A_{\overline G}+I) \ge n, 
\qquad 
\operatorname{rank}(A_G+I)+\operatorname{rank}(A_{\overline G}+I) \ge n+1,
$$
with the equality cases characterized. We further obtain strengthened multiplicative lower bounds under additional structural assumptions.
Finally, we show that the trivial upper bounds
$$
\operatorname{rank}(A_G+I)\cdot\operatorname{rank}(A_{\overline G}+I) \le n^2, 
\qquad 
\operatorname{rank}(A_G+I)+\operatorname{rank}(A_{\overline G}+I) \le 2n
$$
are tight by explicitly constructing, for every $n\ge 4$, graphs $G$ with $\operatorname{rank}(A_G+I)=\operatorname{rank}(A_{\overline G}+I)=n$.
\end{abstract}

\maketitle

\section{Introduction}

We start with some definitions and notation. Throughout this paper we consider only simple graphs, i.e., undirected graphs without loops or multiple edges. Let $G=(V,E)$ be a graph of order $n=|V|$ and size $m=|E|$. The \emph{adjacency matrix} of $G$ is the $n\times n$ matrix $A_G=(a_{ij})$, where $a_{ij}=1$ if the vertices $v_i$ and $v_j$ are adjacent and $a_{ij}=0$ otherwise. For two vertex-disjoint graphs $G_1=(V_1,E_1)$ and $G_2=(V_2,E_2)$, their \emph{union} is defined as
\[
G_1\cup G_2 := (V_1\cup V_2,\; E_1\cup E_2).
\]
In particular, if $G_1$ and $G_2$ are disjoint, then $A_{G_1\cup G_2}$ is block diagonal with blocks $A_{G_1}$ and $A_{G_2}$.

We denote by $K_n$ the complete graph on $n$ vertices, by $K_{a,b}$ the complete bipartite graph with part sizes $a$ and $b$, and by $P_n$ the path on $n$ vertices. We denote by $\overline G$ the complement of a graph $G$. We write $I_n$ for the $n\times n$ identity matrix and $J_n$ for the $n\times n$ all-ones matrix; when the order is clear from the context, the subscript $n$ will be omitted. All matrices are regarded as real matrices, and throughout the paper
\(\rank\) denotes rank over \(\mathbb R\).

The study of inequalities involving a graph parameter $f(G)$ and the same parameter on the 
complement $\overline{G}$ was initiated by Nordhaus and Gaddum in their classical paper~\cite{NG56} on the chromatic number $\chi(G)$ in 1956. They proved that
\[
2\sqrt{n}  \le \chi(G) + \chi(\overline{G})  \le n+1, 
\qquad 
n  \le \chi(G)\cdot \chi(\overline{G})  \le \frac{(n+1)^2}{4}.
\]
Since then, numerous \emph{Nordhaus--Gaddum type bounds} have been established for 
parameters such as treewidth~\cite{JoretWood2012}, 
spectral graph parameters including the sum of squares of positive eigenvalues~\cite{ElphickAouchiche2017} 
and the spectral gap~\cite{KimMadras2025}, 
the rainbow connection number~\cite{ChenLiLian2011}, 
and many others; see the survey in~\cite{AouchicheHansenSurvey}.

Recently, motivated by problems in communication complexity and Rank--Ramsey theory\footnote{A graph $G$ is called \emph{Rank--Ramsey} if both its clique number and the rank of its complement are small, where the rank of a graph means the real rank of its adjacency matrix.},
Beniamini, Linial, and Shraibman~\cite{BeniaminiLinial2024} introduced a new graph parameter called the \emph{complement rank}:
\[
f(G):=\rank(A_G+I).
\]
The matrix $A_G+I$ is the $0$--$1$ matrix which records equality of vertices together with adjacency in $G$. Thus its rank is a linear-algebraic measure of the complexity of this equality-or-adjacency relation. This viewpoint is natural in communication complexity, where the real rank of a communication matrix is a central algebraic measure and is closely connected with the log-rank conjecture~\cite{NisanWigderson1995}. It is also natural in Rank--Ramsey theory: the recently proposed \emph{KRamsey numbers}~\cite{BeniaminiLinial2024} replace the classical Ramsey alternative of finding a large anticlique by a rank condition involving $A_G+I$. In this sense, the complement rank may be viewed as a rank-theoretic analogue of the independence-side condition in classical Ramsey theory.

The complement rank also arises naturally in spectral graph theory, where it is closely related to the independence number $\alpha(G)$, the Lov\'asz theta number $\vartheta(G)$, and the minimum semidefinite rank $\msr(G)$. For a more detailed account of these connections and motivations, we refer the reader to~\cite{BeniaminiLinial2024}.

In this paper we establish Nordhaus--Gaddum type bounds for the complement rank, including both multiplicative and additive upper and lower bounds. 

The rest of this paper is organized as follows. 
In Section~\ref{sec:product1} we prove a Nordhaus--Gaddum type multiplicative lower bound for the complement rank. 
In Section~\ref{sec:sum1} we establish an additive lower bound. Section~\ref{sec:NGproduct2} presents two strengthened versions of the multiplicative lower bound. 
Finally, in Section~\ref{sec:remark1} we discuss the corresponding upper bounds for both the product and the sum.

\section{Product Lower Bound}\label{sec:product1}

We begin with a definition from matrix analysis. For two $m \times n$ matrices 
$A = (a_{ij})$ and $B = (b_{ij})$, the \emph{Hadamard product} (or entrywise product) 
is defined as the $m \times n$ matrix 
\[
A \circ B = (a_{ij} b_{ij}).
\]
A fundamental inequality for the Hadamard product is the following well-known result.

\begin{lem}[Schur Product Theorem]\label{lem:schur}
For any two matrices $A,B \in \mathbb{F}^{n \times n}$ over a field $\mathbb{F}$,  
\[
\rank(A \circ B)  \le \rank(A)\cdot\rank(B).
\]
\end{lem}

Using this lemma, we obtain the following Nordhaus--Gaddum type multiplicative lower bound for the complement rank.

\begin{thm}\label{thm:product-bound}
For every graph $G$ on $n$ vertices,
\[
\rank(A_G+I)\cdot \rank\left(A_{\overline{G}}+I\right)  \ge n,
\]
with equality if and only if either $G$ or $\overline{G}$ is $K_n$.
\end{thm}

\begin{proof}
Set $X := A_G + I$ and $Y := A_{\overline{G}} + I$. By the definition of the complement graph, for $i \neq j$ we have 
$X_{ij}, Y_{ij} \in \{0,1\}$ and 
\[
X_{ij} + Y_{ij} = 1 ,
\]
while $X_{ii} = Y_{ii} = 1$.  
Therefore, the entrywise Hadamard product satisfies
\[
X \circ Y = I_n .
\]
Applying Lemma~\ref{lem:schur}, we obtain
\[
n = \rank(I_n) = \rank(X \circ Y)  \le \rank(X)\cdot\rank(Y)
= \rank(A_G+I)\cdot\rank(A_{\overline{G}}+I),
\]
which proves the inequality.

As for the equality case, it can be deduced from the Theorem~\ref{thm:strengthened-product-1}, which we prove in Section~\ref{sec:NGproduct2}.
\end{proof}

\section{Sum Lower Bound}\label{sec:sum1}
For any $n$-vertex graph $G$, we have $A_{\overline{G}}+I = J - A_G$. Throughout this paper we will use this identity repeatedly.

The next result establishes the Nordhaus--Gaddum type additive lower bound for the complement rank.

\begin{thm}\label{thm:NGsum1}
For every graph $G$ on $n$ vertices,
\[
\rank(A_G+I) + \rank(A_{\overline G}+I) \ge n+1,
\]
with equality if and only if either $G$ or $\overline{G}$ is $K_n$.
\end{thm}
\begin{proof}
Set
\[
X:=A_G+I,\qquad Y:=A_{\overline G}+I=J-A_G,
\qquad S:=X+Y=I+J.
\]
The matrix $S=I+J$ is positive definite, and
\[
S^{-1}=I-\frac{1}{n+1}J.
\]
Define \(B:=S^{-1/2}XS^{-1/2}\). Then $B$ is real symmetric and
\[
S^{-1/2}YS^{-1/2}=I-B.
\]
Since multiplication by an invertible matrix does not change rank,
\begin{equation}\label{eq:rank-sum-via-B}
\rank X+\rank Y=\rank B+\rank(I-B).
\end{equation}
Let $\lambda_1,\ldots,\lambda_n$ be the eigenvalues of $B$. For each eigenvalue, the contribution to the right-hand side of \eqref{eq:rank-sum-via-B} is
\[
\begin{cases}
1, & \lambda_i=0,\\
1, & \lambda_i=1,\\
2, & \lambda_i\notin\{0,1\}.
\end{cases}
\]
Therefore
\begin{equation}\label{eq:rank-sum-n-plus-k}
\rank X+\rank Y=n+k,
\end{equation}
where $k$ is the number of eigenvalues of $B$ outside $\{0,1\}$, counted with multiplicity.

We claim that $k\ge 1$. Suppose, to the contrary, that $k=0$. Then $B(I-B)=0$. Put \(N:=XS^{-1}Y\). Since $Y=S-X$, we have
\[
N=XS^{-1}(S-X)=X-XS^{-1}X,
\]
so $N$ is symmetric. Moreover, \(N=S^{1/2}B(I-B)S^{1/2}\), and hence $N=0$. Let $d_i$ be the degree of vertex $i$ in $G$. The $i$-th column of $X$ has exactly $d_i+1$ entries equal to $1$. Hence
\[
(XS^{-1}X)_{ii}=(d_i+1)-\frac{(d_i+1)^2}{n+1}.
\]
It follows that
\begin{equation}\label{eq:N-diagonal}
N_{ii}
=1-(d_i+1)+\frac{(d_i+1)^2}{n+1}
=\frac{1-d_i(n-1-d_i)}{n+1}.
\end{equation}
Since $N=0$, we know that $d_i(n-1-d_i)=1$ for every $i$. Since $d_i$ and $n-1-d_i$ are nonnegative integers, this forces $d_i=1$ and $n-1-d_i=1$ for every $i$. Thus $n=3$ and every vertex has degree $1$, which is impossible. This contradiction proves $k\ge 1$. By \eqref{eq:rank-sum-n-plus-k},
\[
\rank(A_G+I)+\rank(A_{\overline G}+I)\ge n+1.
\]

It remains to characterize equality. If $G=K_n$, then $A_G+I=J$ and $A_{\overline G}+I=I$, so the rank sum is $1+n=n+1$. The same conclusion holds when $G=\overline{K_n}$.

Conversely, suppose that equality holds. Then \(k=1\), and hence \(N=S^{1/2}B(I-B)S^{1/2}\) is a real symmetric matrix of rank one. We shall show that this is impossible unless \(G\) is complete or empty.

For \(n\le 3\), the assertion is checked directly. Hence we may assume that \(n\ge4\), and that \(G\) is neither complete nor empty. Then \(G\) has a vertex whose degree \(d_i\) satisfies \(1\le d_i\le n-2\). For this vertex, \eqref{eq:N-diagonal} gives $N_{ii}<0$, since \(d_i(n-1-d_i)\ge n-2\ge2\). On the other hand, if a vertex had degree \(0\) or \(n-1\), then \eqref{eq:N-diagonal} would give $N_{ii}=1/(n+1)>0$. This is impossible, because a real symmetric matrix of rank one has the form $N=\lambda uu^{\top}$ with \(\lambda\in\mathbb R\setminus\{0\}\), and therefore all its nonzero diagonal entries \(N_{jj}=\lambda u_j^2\) have the same sign. Consequently every vertex satisfies \(1\le d_j\le n-2\), and all diagonal entries of \(N\) are negative. It follows from \(N=\lambda uu^{\top}\) that \(\lambda<0\). Therefore, for every \(z\in\mathbb R^n\),
\[
        z^{\top}Nz=\lambda (u^{\top}z)^2\le0.
\]
Thus \(N\) is negative semidefinite.

By Sylvester's law of inertia, $B(I-B)$ also has exactly one negative eigenvalue and all its other eigenvalues are zero. Since $B$ has exactly one eigenvalue outside $\{0,1\}$, say $\theta$, we have \(\theta(1-\theta)<0\). Thus either $\theta>1$ or $\theta<0$. If $\theta>1$, then $B$ is positive semidefinite, and hence
\[
X=S^{1/2}BS^{1/2}
\]
is positive semidefinite. If $\theta<0$, then $I-B$ is positive semidefinite, and hence
\[
Y=S^{1/2}(I-B)S^{1/2}
\]
is positive semidefinite.

We use the following elementary observation: if $H$ is a graph such that $A_H+I$ is positive semidefinite, then $H$ is a disjoint union of complete graphs. Indeed, if a connected component of $H$ is not complete, then it contains an induced $P_3$. On the three vertices of this induced path, the corresponding principal submatrix of $A_H+I$ is
\[
\begin{pmatrix}
1&1&0\\
1&1&1\\
0&1&1
\end{pmatrix},
\]
whose determinant is $-1$, contradicting positive semidefiniteness.

Suppose first that $X$ is positive semidefinite. Then
\[
G=K_{s_1}\cup\cdots\cup K_{s_q}.
\]
Since every vertex has degree between $1$ and $n-2$, we have $q\ge2$ and $s_j\ge2$ for all $j$. Thus $\rank X=q$. The complement $\overline G$ is a complete $q$-partite graph. Let \(V_1,\ldots,V_q\) be the parts of \(\overline G\), with
\(|V_j|=s_j\). Define
\[
W:=\left\{x\in\mathbb R^n:\ \sum_{v\in V_j}x_v=0
\text{ for every }j=1,\ldots,q\right\}.
\]
Then \(\dim W=\sum_{j=1}^q(s_j-1)=n-q\). We claim that \(Y\) acts as the identity on \(W\). Indeed, if \(x\in W\) and
\(v\in V_a\), then, since \(\overline G\) is complete \(q\)-partite,
\[
(Yx)_v
=x_v+\sum_{u\sim_{\overline G} v}x_u
=x_v+\sum_{b\ne a}\sum_{u\in V_b}x_u
=x_v.
\]
Thus \(Yx=x\) for every \(x\in W\), as claimed. Let
\[
U:=\{x\in\mathbb R^n:\ x \text{ is constant on each } V_j\}.
\]
Then \(\dim U=q\), and \(\mathbb R^n=W\oplus U\). Indeed, if \(x\in W\cap U\), then \(x\) is equal to some constant \(a_j\) on
\(V_j\), and hence
\[
0=\sum_{v\in V_j}x_v=s_j a_j
\]
for every \(j\), so \(x=0\). Since \(\dim W+\dim U=(n-q)+q=n\), the direct-sum
decomposition follows.

Moreover, \(U\) is \(Y\)-invariant. Indeed, if \(x\in U\) has value \(a_j\)
on \(V_j\), then for \(v\in V_i\),
\[
(Yx)_v
=a_i+\sum_{\ell\ne i}s_\ell a_\ell,
\]
which depends only on \(i\). Thus \(Yx\in U\).

We now show directly that \(Y|_U\) has rank at least \(2\). Choose two distinct
indices \(p,r\). For \(j=1,\ldots,q\), let \(\mathbf 1_{V_j}\) denote the
indicator vector of \(V_j\). Since \(\overline G\) is complete \(q\)-partite,
the vector \(Y\mathbf 1_{V_j}\) has value \(1\) on \(V_j\) and value \(s_j\)
on every other part. Suppose that
\[
\alpha Y\mathbf 1_{V_p}+\beta Y\mathbf 1_{V_r}=0.
\]
Looking at a coordinate in \(V_p\) and a coordinate in \(V_r\), respectively, gives
\[
\alpha+\beta s_r=0,\qquad \alpha s_p+\beta=0.
\]
The determinant of this linear system is \(1-s_ps_r\ne0\), because \(s_p,s_r\ge2\). Hence \(\alpha=\beta=0\). Therefore
\(Y\mathbf 1_{V_p}\) and \(Y\mathbf 1_{V_r}\) are linearly independent, and so
\[
\rank(Y|_U)\ge2.
\]
Since \(W\) and \(U\) are both \(Y\)-invariant and
\(\mathbb R^n=W\oplus U\), while \(Y|_W=I_W\), we obtain
\[
\rank Y=\rank(Y|_W)+\rank(Y|_U)\ge (n-q)+2.
\]
Consequently,
\[
\rank X+\rank Y\ge q+(n-q+2)=n+2,
\]
contradicting equality. If instead $Y$ is positive semidefinite, the same argument applied to $\overline G$ gives the same contradiction. Therefore equality can occur only when $G=K_n$ or $G=\overline{K_n}$.
\end{proof}

\section{Strengthened Product Lower Bounds}\label{sec:NGproduct2}
To establish stronger multiplicative lower bounds, we begin with a simple but useful classification of graphs with very small complement rank.

\begin{lem}\label{lem:rank2}
Let $G$ be a graph on $n$ vertices. Then,
\begin{enumerate}
    \item $\rank(A_{G}+I)=1$ if and only if $G$ is the complete graph $K_n$.
    \item $\rank(A_G+I)=2$ if and only if $\overline{G}$ is the complete bipartite graph $K_{a, n-a}$ with $1 \le a \le \lfloor n/2 \rfloor$.
\end{enumerate}
\end{lem}

\begin{proof}
\emph{(1)} If \(G=K_n\), then \(A_G+I=J_n\), so \(\rank(A_G+I)=1\).
Conversely, suppose that \(\rank(A_G+I)=1\). If \(G\ne K_n\), then there exist
two distinct nonadjacent vertices \(u,v\). The principal submatrix of \(A_G+I\)
on \(\{u,v\}\) is
\[
\begin{pmatrix}
1&0\\
0&1
\end{pmatrix},
\]
which has rank \(2\), a contradiction. Hence \(G=K_n\).

\smallskip
\emph{(2)} By \cite[Lemma~2.14]{BeniaminiLinial2024}, if $\overline{G}$ is connected, then we know that $\rank(A_G+I) = 2$ if and only if $\overline{G}$ is a blowup of a $2$-clique. 

If $\overline{G}$ is not connected, then $G$ must be connected. By (1), we can assume $G$ is not complete.
Then there exist nonadjacent vertices $u,v$. Take a shortest $u$--$v$ path
$u=x_0,x_1,x_2,\dots,x_k=v$ with $k\ge2$. By minimality, $x_0$ and $x_2$ are not adjacent, so
the induced subgraph on $\{x_0,x_1,x_2\}$ is a path $P_3$. The $3\times3$ principal submatrix of $A_G+I$ on $\{x_0,x_1,x_2\}$ is
\[
\begin{pmatrix}
1&1&0\\
1&1&1\\
0&1&1
\end{pmatrix},
\]
whose determinant equals $-1\neq0$. Hence this principal minor is nonsingular, so
$\rank(A_G+I)\ge3$, ruling out $\rank(A_G+I)=2$.
\end{proof}

With Lemma~\ref{lem:rank2} in hand, we can now derive our first strengthened product lower bound, which improves the general product inequality whenever the graph is neither complete nor empty.

\begin{thm}\label{thm:strengthened-product-1}
Let $G$ be a simple graph on $n$ vertices that is neither $K_n$ nor $\overline{K_n}$.
Then
\[
\rank(A_G+I)\cdot\rank(A_{\overline G}+I) \ge 2n.
\]
Moreover, equality holds if and only if either $G$ or $\overline G$ is $K_{a, n-a}$ with $1 \le a \le \lfloor n/2 \rfloor$.
\end{thm}

\begin{proof}
Write $X:=A_G+I$ and $Y:=A_{\overline G}+I$. Since $G$ is neither $K_n$ nor $\overline{K_n}$, by Theorem~\ref{thm:NGsum1} we know that 
\begin{equation}\label{eq:no-equality-sum-lower-bound-1}
\rank X + \rank Y  \ge n+2.
\end{equation}

We divide into cases according to $\rank X$ and $\rank Y$.

\medskip
\emph{Case 1:} $\rank Y=1$ (the case $\rank X = 1$ is symmetric). By Lemma~\ref{lem:rank2}, this forces $\overline{G}=K_n$, contradicting the assumption. Hence this case cannot occur.

\medskip
\emph{Case 2:} $\rank Y=2$ (the case $\rank X = 2$ is symmetric). Then Lemma~\ref{lem:rank2} implies $G=K_{a,n-a}$ for some $1\le a\le n-1$, so $\overline{G}=K_a\cup K_{n-a}$.  
The adjacency spectrum of $K_{a,n-a}$ is $\{\sqrt{a(n-a)},-\sqrt{a(n-a)},0^{(n-2)}\}$, hence
$A_G+I$ has eigenvalues $\{1\pm\sqrt{a(n-a)},1^{(n-2)}\}$, which are all nonzero for $a(n-a)\ge 2$.
Thus $\rank X=n$ and $\rank Y=2$, yielding
\[
\rank X\cdot\rank Y = 2n.
\]
(The degenerate case $a(n-a)=1$ corresponds to $n=2$, which is excluded by hypothesis.)

\medskip
\emph{Case 3:} $\rank X\ge 3$ and $\rank Y\ge 3$.
Then, together with \eqref{eq:no-equality-sum-lower-bound-1},
\[
\rank X + \rank Y  \ge n+2,\qquad \rank X,\rank Y\ge 3.
\]
Fix $n\ge 4$. The product $rs$ (with integers $r,s\ge 3$ and $r+s\ge n+2$) is minimized at $(r,s)=(3,n-1)$, giving
\[
\rank X\cdot\rank Y  \ge 3(n-1) \ge 2n.
\]
For the finitely many small orders $n\in\{1,2,3\}$, the inequality is easily checked directly.

\medskip
Combining the three cases establishes the lower bound. The equality case occurs precisely in Case~2 (or its symmetric counterpart), i.e., when $G$ or $\overline G$ is a complete bipartite graph. This completes the proof.
\end{proof}

The next result gives an even stronger bound under the additional assumption that the graph and its complement are not complete bipartite. This yields a second strengthened version of the product inequality.

\begin{thm}\label{thm:strengthened-product-2}
Let $G$ be a graph on $n$ vertices that is neither $K_n$ nor $\overline{K_n}$ nor $K_{a,n-a}$ nor $\overline{K_{a,n-a}}$ for any $1\le a\le \lfloor n/2 \rfloor$. Then
\[
\rank(A_G + I) \cdot \rank(A_{\overline{G}} + I)  \ge 3(n-1).
\]
\end{thm}
\begin{proof}
Write $X := A_G + I$ and $Y := A_{\overline{G}} + I$. 
Since $G$ is neither $K_n$ nor $\overline{K_n}$ nor $K_{a,n-a}$ nor $\overline{K_{a,n-a}}$ for any $1\le a\le \lfloor n/2 \rfloor$, Lemma~\ref{lem:rank2} ensures that $\rank X \ge 3$ and $\rank Y \ge 3$.

We are in exactly the setting of Case~3 in the proof of Theorem~\ref{thm:strengthened-product-1}. 
Minimizing the product under these constraints gives
\[
\rank X \cdot \rank Y \ge 3(n-1),
\]
as required.
\end{proof}
\begin{remark}
For \(n\ge4\), the bound in Theorem~\ref{thm:strengthened-product-2} is best possible. Equality holds, for instance, when $G=K_2\cup (n-2)K_1$ or $G=K_{n-2}\cup 2K_1$.
\end{remark}

\section{Product and Sum Upper Bounds}\label{sec:remark1}

Throughout the previous sections we have established various Nordhaus--Gaddum 
type lower bounds for the complement rank, both in multiplicative 
and additive form. It is also natural to ask about the corresponding upper bounds. In this regard the situation is much simpler.

\begin{prop}\label{prop:trivial-upper}
For any graph $G$ of order $n$,
\[
\rank(A_G + I) \cdot \rank(A_{\overline{G}} + I) \le n^2, \qquad 
\rank(A_G + I) + \rank(A_{\overline{G}} + I) \le 2n,
\]
with equality in either inequality if and only if 
$\rank(A_G + I)=\rank(A_{\overline{G}} + I)=n$.
\end{prop}

\begin{proof}
The inequalities are immediate from the general bounds 
$0 \le \rank(A_G + I), \rank(A_{\overline{G}} + I) \le n$. 
If equality holds in the product inequality, then 
$\rank(A_G + I)=\rank(A_{\overline{G}} + I)=n$. Similarly, if equality holds in the sum inequality, then
$\rank(A_G + I)=\rank(A_{\overline{G}} + I)=n$. The converse is obvious.
\end{proof}

Thus the only way to attain the trivial upper bounds is to find graphs $G$ such that both $A_G+I$ and $A_{\overline{G}}+I$ are nonsingular. In particular, the existence of such graphs for each $n$ is precisely the content of our next theorem.

\begin{thm}\label{thm:exist-fullrank}
For every integer $n \ge 4$, there exists a graph $G_n$ of order $n$ such that
\[
\rank(A_{G_n}+I_n) = \rank(A_{\overline{G_n}}+I_n) = n.
\]
\end{thm}

Before proving Theorem~\ref{thm:exist-fullrank}, we first establish several preliminary results that will be needed in the argument.

\begin{lem}[Path spectrum]\label{lem:path-spectrum}
Let $A_{P_m}$ be the adjacency matrix of the path $P_m$. Its eigenvalues are
\[
\lambda_k = 2\cos \frac{k\pi}{m+1},\qquad k=1,2,\dots,m.
\]
In particular,
\[
\rank(A_{P_m}+I_m) =
\begin{cases}
m, & m\not\equiv 2 \pmod{3},\\
m-1, & m\equiv 2 \pmod{3}.
\end{cases}
\]
\end{lem}

\begin{proof}
The path spectrum is classical (see, e.g.~\cite[Section~2.6.7]{Cvetkovic1980}). Then $-1\in\sigma(A_{P_m})$ if and only if $2\cos\theta=-1$ for some $\theta=\frac{k\pi}{m+1}$, i.e., 
$\cos\theta=-\frac12$, hence $\theta=\frac{2\pi}{3}$ and $3k=2(m+1)$, which is possible if and only if $m\equiv 2\pmod{3}$.
Therefore $A_{P_m}+I_m$ is invertible if and only if $m\not\equiv 2\pmod{3}$.
\end{proof}

\begin{lem}\label{lem:JminusA-even}
Let $m\ge 2$ be even. Then \(\rank(J_m-A_{P_m})=m\).
\end{lem}

\begin{proof}
Let $A:=A_{P_m}$. Consider $(J_m-A)x=0$. Writing $\mathbf{1}$ for the all-ones column vector and $c:=\mathbf{1}^\top x$, we have
\[
Ax=J_mx=c\,\mathbf{1}.
\]
Expanding $Ax=c\mathbf{1}$ in coordinates using the tridiagonal structure of $A$ gives
\begin{equation}\label{eq:coord-rec}
\begin{cases}
x_2=c,\\
x_{i-1}+x_{i+1}=c,& 2\le i\le m-1,\\
x_{m-1}=c.
\end{cases}
\end{equation}
Let $y_i:=x_i-\frac{c}{2}$. Then $y_{i+1}+y_{i-1}=0$ for $2\le i\le m-1$, with $y_2=\frac{c}{2}$ and $y_{m-1}=\frac{c}{2}$.
The homogeneous recurrence has period $4$: if $y_1=a$, $y_2=b$, then $y_3=-a$, $y_4=-b$, $y_5=a$, $y_6=b$, etc. Since $m$ is even, for the odd index $m-1$ we have
\begin{equation}\label{eq:ym1}
y_{m-1}=\begin{cases}
-a,& m\equiv 0\pmod 4,\\
\ \ a,& m\equiv 2\pmod 4,
\end{cases}
\end{equation}
while $y_2=b$ by definition. Using the boundary conditions $y_2=\dfrac{c}{2}$ and $y_{m-1}=\dfrac{c}{2}$, we get
\begin{equation}\label{eq:a-b-from-c}
(a,b)=
\begin{cases}
\left(-\dfrac{c}{2}, \dfrac{c}{2}\right),& m\equiv 0\pmod 4,\\[6pt]
\left( \dfrac{c}{2}, \dfrac{c}{2}\right),& m\equiv 2\pmod 4.
\end{cases}
\end{equation}
One can also check that 
\begin{equation}\label{eq:ym11121}
\sum_{i=1}^m y_i=\begin{cases}
0,& m\equiv 0\pmod 4,\\
a+b,& m\equiv 2\pmod 4.
\end{cases}
\end{equation}
Using $c=\sum_{i=1}^m x_i=\frac{m}{2}c+\sum_{i=1}^m y_i$, one obtains in both subcases that $c=0$, hence $x_2=x_{m-1}=0$ and the recurrence forces $x\equiv 0$.
Thus $\ker(J_m-A)=\{0\}$ and $J_m-A$ is invertible.
\end{proof}

%================= CASE I =================
\begin{prop}\label{prop:even-case}
Let $n\ge 4$ with $n\equiv 0$ or $4\pmod{6}$, and set $G=P_n$.
Then
\[
\rank(A_G+I_n) = \rank(A_{\overline{G}}+I_n) = n.
\]
\end{prop}

\begin{proof}
By Lemma~\ref{lem:path-spectrum}, since $n\not\equiv 2\pmod{3}$ we have $\rank(A_{P_n}+I_n)=n$.
For the complement side, we need $\rank(J_n-A_{P_n})=n$,
which holds by Lemma~\ref{lem:JminusA-even} for even $n$.
\end{proof}

%================= CASE II =================
\begin{prop}\label{prop:odd-case}
Let $n\ge 4$ with $n\equiv 1$ or $5\pmod{6}$, and set $G=P_{n-1}\cup K_1$.
Then
\[
\rank(A_G+I_n) = \rank(A_{\overline{G}}+I_n) = n.
\]
\end{prop}

\begin{proof}
Write $m:=n-1$. Then $m\equiv 0$ or $4\pmod{6}$. Since $G$ is a disjoint union, 
\[
A_G+I_n = \begin{pmatrix} A_{P_m}+I_m & 0 \\ 0 & 1 \end{pmatrix},
\]
hence $\rank(A_G+I_n)=\rank(A_{P_m}+I_m)+1 = m+1 = n$ by Proposition~\ref{prop:even-case}.
For the complement side,
\[
A_{\overline{G}}+I_n = J_n - A_G = 
\begin{pmatrix}
J_m-A_{P_m} & \mathbf{1} \\
\mathbf{1}^\top & 1
\end{pmatrix}.
\]
By Lemma~\ref{lem:JminusA-even}, $J_m-A_{P_m}$ is invertible. 
Write $M:=J_m-A_{P_m}$ and let $\mathbf{1}\in\mathbb{R}^m$ denote the all--ones column vector.
Consider the block matrix
\[
B =
\begin{pmatrix}
M & \mathbf{1}\\
\mathbf{1}^\top & 1
\end{pmatrix}.
\]
We claim that $B$ is invertible. Indeed, suppose $B\binom{u}{w}=0$ with $u\in\mathbb{R}^m$, $w\in\mathbb{R}$. 
Then
\[
Mu+w \mathbf{1}=0,\qquad \mathbf{1}^\top u + w = 0.
\]
From the second equation $w=-\mathbf{1}^\top u$, and substituting into the first gives
\[
(J_m-A_{P_m})u-(\mathbf{1}^\top u) \mathbf{1}=0
\ \Longleftrightarrow\
J_m u - A_{P_m} u - (\mathbf{1}^\top u) \mathbf{1}=0.
\]
Since $J_m u=(\mathbf{1}^\top u) \mathbf{1}$, the two terms cancel and we obtain
\[
A_{P_m}u=0.
\]
Because $m$ is even, $A_{P_m}$ has no zero eigenvalue (by Lemma~\ref{lem:path-spectrum}) and is therefore invertible; hence $u=0$, and then $w=-\mathbf{1}^\top u=0$. Thus $\ker B=\{0\}$ and $B$ is invertible, i.e., $\rank\left(A_{\overline{G}}+I_n\right)=n$.
\end{proof}

\begin{prop}\label{prop:n2mod6}
Let $n\ge 4$ with $n\equiv 2\pmod 6$, and set \(G = P_4 \cup P_{n-4}\). Then
\[
\rank(A_G+I_n) = \rank(A_{\overline{G}}+I_n) = n.
\]
\end{prop}

\begin{proof}
Write $A_m:=A_{P_m}$. Since $G$ is a disjoint union, 
\[
A_G+I_n = (A_4+I_4) \oplus (A_{n-4}+I_{n-4}).
\] Since $4\equiv 1\pmod 3$ and $n-4\equiv 1\pmod 3$, by Lemma~\ref{lem:path-spectrum} we obtain
\[
\rank(A_G+I_n)=\rank(A_4+I_4)+\rank(A_{n-4}+I_{n-4})=4+(n-4)=n.
\]

It remains to prove $\rank(A_{\overline G}+I_n)=n$, i.e., that $J_n-A_G$ is invertible.
Write $x=(u,v)^\top$, where $u\in\mathbb{R}^4$ and $v\in\mathbb{R}^{n-4}$, according to $V(G)=V(P_4) \cup V(P_{n-4})$ and set $c:=\mathbf{1}^\top x$.
Then $(J_n-A_G)x=0$ is equivalent to
\[
A_4u=c \mathbf{1}_4,\qquad A_{n-4}v=c\,\mathbf{1}_{n-4}.
\]
Summing coordinates on each component and using the path recurrence (as in the proof of
Lemma~\ref{lem:JminusA-even}) one obtains
\[
\sum_{i=1}^{4} u_i=2c,\qquad 
\sum_{i=1}^{n-4} v_i=
\begin{cases}
\frac{n-4}{2}\,c,& n\equiv 0\pmod 4,\\
\left(\frac{n-4}{2}+1\right)c,& n\equiv 2\pmod 4.
\end{cases}
\]
Hence
\[
c=\mathbf{1}^\top x=\sum_{i=1}^{4} u_i+\sum_{i=1}^{n-4} v_i=
\begin{cases}
\frac{n}{2}\,c,& n\equiv 0\pmod 4,\\
\left(\frac{n}{2}+1\right)c,& n\equiv 2\pmod 4,
\end{cases}
\]
which forces $c=0$ (since $n\ge 4$). Consequently $A_4u=0$ and $A_{n-4}v=0$, and because both
$A_4$ and $A_{n-4}$ are invertible, we get $u=v=0$.
Thus $\ker(J_n-A_G)=\{0\}$, proving that $J_n-A_G$ is invertible and
$\rank(A_{\overline G}+I_n)=n$.
\end{proof}

\begin{prop}\label{prop:n3mod6}
Let $n\ge 4$ with $n\equiv 3\pmod 6$, and set \(G = P_4 \cup P_{n-5} \cup K_1\). Then
\[
\rank(A_G+I_n) = \rank(A_{\overline{G}}+I_n) = n .
\]
\end{prop}

\begin{proof}
Write $A_m:=A_{P_m}$. Since $G$ is a disjoint union,
\[
A_G+I_n=(A_4+I_4)\oplus(A_{n-5}+I_{n-5})\oplus(1).
\]
Here $4\equiv1\pmod 3$ and $n-5\equiv1\pmod 3$, by Lemma~\ref{lem:path-spectrum} we obtain 
\[
\rank(A_G+I_n)=\rank(A_4+I_4)+\rank(A_{n-5}+I_{n-5})+1=4+(n-5)+1=n.
\]

For the complement, $(J_n-A_G)x=0$ implies $A_Gx=c\mathbf{1}$ with $c=\mathbf{1}^\top x$.
Splitting $x=(u,v,w)$, where $u\in\mathbb{R}^4$, $v\in\mathbb{R}^{n-5}$, and $w\in\mathbb{R}$, along the three components yields
\[
A_4u=c\mathbf{1}_4,\quad A_{n-5}v=c\mathbf{1}_{n-5},\quad 0=c.
\]
Hence $c=0$, and since $A_4$ and $A_{n-5}$ are invertible, we get $u=v=0$ and also $w=0$.
Thus $\ker(J_n-A_G)=\{0\}$ and $\rank\left(A_{\overline G}+I_n\right)=n$.\end{proof}

We are now ready to present the following.
\begin{proof}[Proof of Theorem~\ref{thm:exist-fullrank}]
Define $G_n$ by cases according to $n\bmod 6$:
\[
G_n \;:=\;
\begin{cases}
P_n, & n\equiv 0,4 \pmod 6,\\
P_{n-1} \cup K_1, & n\equiv 1,5 \pmod 6,\\
P_4 \cup P_{n-4}, & n\equiv 2 \pmod 6,\\
P_4 \cup P_{n-5} \cup K_1, & n \equiv 3 \pmod 6.
\end{cases}
\]
By Propositions~\ref{prop:even-case}, \ref{prop:odd-case}, \ref{prop:n2mod6}, and~\ref{prop:n3mod6}, 
each case satisfies \(\rank(A_{G_n}+I_n)=\rank(A_{\overline{G_n}}+I_n)=n\). This completes the proof.\end{proof}

\section*{Acknowledgements}
The author thanks Clive Elphick for many valuable suggestions, which have significantly improved an earlier version of this paper. The author is also grateful to Jie Ma and Shengtong Zhang for their helpful comments. The author also thanks an anonymous referee for pointing out an issue in the proof of Theorem~\ref{thm:NGsum1} in an earlier version.

\end{document}